\documentclass[letterpaper, 10 pt, conference]{ieeeconf}
\IEEEoverridecommandlockouts
\let\labelindent\relax
\usepackage{enumitem}
\usepackage{balance}
\usepackage[font={small}]{caption}
\usepackage{subcaption}
\usepackage{array}
\usepackage{textcomp}
\usepackage{mathtools, nccmath}
\usepackage{graphicx}
\usepackage{amsfonts}
\usepackage{amsmath}
\usepackage{amssymb}
\usepackage{algorithm}
\usepackage{algorithmic}
\usepackage{hyperref}
\usepackage{tikz}
\usepackage{arydshln}
\usepackage{multirow}
\usepackage{bm}
\usepackage{epstopdf}
\usepackage{cite}
\usepackage{siunitx}

\usepackage{amsthm}

\newtheorem{theorem}{Theorem}
\newtheorem{lemma}{Lemma}

\newtheorem{remark}{Remark}
\theoremstyle{definition}
\newtheorem{definition}{Definition}
\newtheorem{problem}{Problem}

\title{\LARGE \bf
    Auxiliary-Variable Adaptive Control Lyapunov Barrier Functions for Spatio-Temporally Constrained Safety-Critical Applications}

\author{Shuo Liu$^{1}$, Wei Xiao$^{2}$ and Calin A. Belta$^{3}$
\thanks{This work was supported in part by the NSF under grant IIS-2024606 at Boston University.}
\thanks{$^{1}$S. Liu is with the department of Mechanical Engineering, Boston
University, Brookline, MA, USA. 
        {\tt\small liushuo@bu.edu}}%
\thanks{$^{2}$W. Xiao is with the Computer Science and Artificial Intelligence Lab, Massachusetts Institute of Technology, Cambridge, MA, USA 
        {\tt\small weixy@mit.edu}}%
\thanks{$^{3}$C. Belta is with the Department of Electrical and Computer Engineering and the Department of Computer Science, University of Maryland, College Park, MD, USA 
        {\tt\small cbelta@umd.edu}}%
}

\begin{document} 
\maketitle

\begin{abstract}
Recent work has shown that stabilizing an affine control system while optimizing a quadratic cost subject to state and control constraints can be mapped to a sequence of Quadratic Programs (QPs) using Control Barrier Functions (CBFs) and Control Lyapunov Functions (CLFs). One of the main challenges in this method is that the QPs could easily become infeasible under safety and spatio-temporal constraints with tight control bounds.  In our own recent work, we defined Auxiliary-Variable Adaptive CBFs (AVCBFs) to improve the feasibility of the CBF-based QP, while avoiding extensive parameter tuning. In this paper, we consider spatio-temporal constraints as finite-time reachability requirements. In order to satisfy these requirements, we generalize AVCBFs to Auxiliary-Variable Adaptive Control Lyapunov Barrier Functions (AVCLBFs) that work for systems and constraints with arbitrary relative
degrees. We show that our method has fewer conflicts with safety and input constraints, and outperforms the state of the art in term of adaptivity and feasibility in solving the QP. We illustrate our approach on an optimal control problem for a unicycle.
\end{abstract}

\section{Introduction}
\label{sec:Introduction}

Safety is one of the primary concerns in
the design and operation of autonomous systems. A considerable number of studies have incorporated safety directly into optimal control problems as constraints, employing Barrier Functions (BF) and Control Barrier Functions (CBF). BFs are Lyapunov-like functions \cite{tee2009barrier} whose use can be traced back to optimization problems \cite{boyd2004convex}. They have been utilized to prove set invariance \cite{aubin2011viability}, \cite{prajna2007framework} to derive multi-objective control \cite{panagou2013multi}, \cite{wang2016multi}, and to control multi-robot systems
\cite{glotfelter2017nonsmooth}. 

CBFs are extensions of BFs used to enforce safety, i.e., rendering a set forward invariant, for an affine control system. It was proved in \cite{ames2016control} that if a CBF for a safe set satisfies Lyapunov-like conditions, then this set is forward invariant and safety is guaranteed. It has also been shown that stabilizing an affine control system to admissible states, while minimizing a quadratic cost subject to state and control constraints, can be mapped to a sequence of Quadratic Programs (QPs) \cite{ames2016control} by unifying CBFs and Control Lyapunov Functions (CLFs) \cite{ames2012control}. 
In its initial formulation, this approach, which in this paper we will refer to as CBF-CLF-QP, is applicable exclusively to safety constraints that possess a relative degree of one. 
Exponential CBFs were developed to extend the methodology to encompass constraints with higher relative degrees \cite{nguyen2016exponential}. A broader category known as High-Order CBFs (HOCBFs) was introduced in \cite{xiao2021high}. The CBF-CLF-QP method has been widely used to enforce safety in many applications, including rehabilitative system control \cite{isaly2020zeroing}, adaptive cruise control \cite{ames2016control}, humanoid robot walking \cite{khazoom2022humanoid} and robot swarming \cite{cavorsi2022multi}. 

Recent works also employed CBFs or CLFs to enforce spatio-temporal specifications on the system trajectories. In \cite{lindemann2018control,liu2021recurrent,lindemann2019control,garg2019control}, Signal Temporal Logic (STL) was used as a specification language. Nevertheless, these papers only considered constraints with relative degree one. This limitation is not an issue for methods that use Recurrent Neural Network (RNN) controllers in \cite{liu2021recurrent}. However, training such controllers can be computationally very intensive. The input constraints were not considered at all in \cite{lindemann2018control}. Although they were taken into account in \cite{garg2019control,garg2022fixed,xiao2021high2} for constraints with relative degree one, and in \cite{xiao2021high2} for constraints with high relative degrees,
the CBF-CLF-QP approaches in these papers are likely to become infeasible due to conflicts between safety constraints, spatio-temporal constraints, and (tight) control bounds.
There are several approaches that aim to enhance the feasibility of the CBF-QP or CBF-CLF-QP based methods \cite{liu2023iterative,gurriet2018online,singletary2019online,chen2021backup,squires2018constructive,breeden2021high,xiao2021adaptive,liu2023auxiliary,xiao2022sufficient,liu2023feasibility}, However, these papers do not considere the improvement of feasibility under spatio-temporal constraints.

In this paper, the spatio-temporal constraints are considered as finite-time reachability constraints, whose objective is to guarantee the reachability of the states of a system to target states within a user-defined time (see \cite{garg2020prescribed}). To accommodate safety and finite-time reachability over nonlinear state constraints for high relative degree systems,
we propose Auxiliary-Variable Adaptive Control Lyapunov Barrier Functions
(AVCLBFs). Specifically, we formulate finite-time reachability as AVCLBF constraints with high relative degree. 
We show that AVCLBF constraints induce fewer conflicts with safety and input constraints than those in the works mentioned above, and thus the  feasibility of CBF-CLF-QP is enhanced. 
We demonstrate the effectiveness of the proposed method on a unicycle model with tight control bounds, and compare it to the state of the art. The results show that our proposed approach is able to enforce finite-time reachability in the presence of safety and control input constraints under various map configurations.

\section{Definitions and Preliminaries}
\label{sec:Preliminaries}

Consider an affine control system of the form
\begin{equation}
\label{eq:affine-control-system}
\dot{\boldsymbol{x}}=f(\boldsymbol{x})+g(\boldsymbol{x})\boldsymbol{u},
\end{equation}
 where $\boldsymbol{x}\in \mathbb{R}^{n}, f:\mathbb{R}^{n}\to\mathbb{R}^{n}$ and $g:\mathbb{R}^{n}\to\mathbb{R}^{n\times q}$ are locally Lipschitz, and $\boldsymbol{u}\in \mathcal U\subset \mathbb{R}^{q}$, where $\mathcal U$ denotes the control limitation set, which is assumed to be in the form: 
\begin{equation}
\label{eq:control-constraint}
\mathcal U \coloneqq \{\boldsymbol{u}\in \mathbb{R}^{q}:\boldsymbol{u}_{min}\le \boldsymbol{u} \le \boldsymbol{u}_{max} \}, 
\end{equation}
with $\boldsymbol{u}_{min},\boldsymbol{u}_{max}\in \mathbb{R}^{q}$ (vector inequalities are interpreted componentwise).
 
\begin{definition}[Class $\kappa$ and Extended Class $\kappa$ function~\cite{Khalil:1173048}]
\label{def:class-k-f}
A continuous function $\alpha: \mathbb{R} \to \mathbb{R}$ is called an extended class $\kappa$ function if it is strictly increasing and $\alpha(0)=0.$ If $\alpha: [0,a)\to[0,+\infty]$ strictly increases and $a>0,$ then it is class $\kappa$ function.
\end{definition}

\begin{definition}
\label{def:forward-inv}
A set $\mathcal C\subset \mathbb{R}^{n}$ is forward invariant for system \eqref{eq:affine-control-system} if its solutions for some $\boldsymbol{u} \in \mathcal U$ starting from any $\boldsymbol{x}(0) \in \mathcal C$ satisfy $\boldsymbol{x}(t) \in \mathcal C, \forall t \ge 0.$
\end{definition}

\begin{definition}
\label{def:relative-degree}
The relative degree of a differentiable function $b:\mathbb{R}^{n}\to\mathbb{R}$ is the minimum number of times we need to differentiate it along dynamics \eqref{eq:affine-control-system} until any component of $\boldsymbol{u}$ explicitly shows in the corresponding derivative. 
\end{definition}

In this paper,  a \textbf{safety requirement} is defined as $b(\boldsymbol{x})\ge0$, and \textbf{safety} is the forward invariance of the set $\mathcal C\coloneqq \{\boldsymbol{x}\in\mathbb{R}^{n}:b(\boldsymbol{x})\ge 0\}$. The relative degree of function $b$ is also referred to as the relative degree of safety requirement $b(\boldsymbol{x}) \ge 0$. 

\subsection{High-Order Control Barrier Functions (HOCBFs)}

For a requirement $b(\boldsymbol{x})\ge0$ with relative degree $m$ and $\psi_{0}(\boldsymbol{x})\coloneqq b(\boldsymbol{x}),$ we define a sequence of functions $\psi_{i}:\mathbb{R}^{n}\to\mathbb{R},\ i\in \{1,...,m\}$ as
\begin{equation}
\label{eq:sequence-f1}
\psi_{i}(\boldsymbol{x})\coloneqq\dot{\psi}_{i-1}(\boldsymbol{x})+\alpha_{i}(\psi_{i-1}(\boldsymbol{x})),\ i\in \{1,...,m\}, 
\end{equation}
where $\alpha_{i}(\cdot ),\ i\in \{1,...,m\}$ denotes a $(m-i)^{th}$ order differentiable class $\kappa$ function. We further set up a sequence of sets $\mathcal C_{i}$ based on \eqref{eq:sequence-f1} as
\begin{equation}
\label{eq:sequence-set1}
\mathcal C_{i}\coloneqq \{\boldsymbol{x}\in\mathbb{R}^{n}:\psi_{i}(\boldsymbol{x})\ge 0\}, \ i\in \{0,...,m-1\}. 
\end{equation}
\begin{definition}[HOCBF~\cite{xiao2021high}]
\label{def:HOCBF}
Let $\psi_{i}(\boldsymbol{x}),\ i\in \{1,...,m\}$ be defined by \eqref{eq:sequence-f1} and $\mathcal C_{i},\ i\in \{0,...,m-1\}$ be defined by \eqref{eq:sequence-set1}. A function $b:\mathbb{R}^{n}\to\mathbb{R}$ is a High-Order Control Barrier Function (HOCBF) with relative degree $m$ for system \eqref{eq:affine-control-system} if there exist $(m-i)^{th}$ order differentiable class $\kappa$ functions $\alpha_{i},\ i\in \{1,...,m\}$ such that
\begin{equation}
\label{eq:highest-HOCBF}
\begin{split}
\sup_{\boldsymbol{u}\in \mathcal U}[L_{f}^{m}b(\boldsymbol{x})+L_{g}L_{f}^{m-1}b(\boldsymbol{x})\boldsymbol{u}+O(b(\boldsymbol{x}))
+\\
\alpha_{m}(\psi_{m-1}(\boldsymbol{x}))]\ge 0,
\end{split}
\end{equation}
$\forall \boldsymbol{x}\in \mathcal C_{0}\cap,...,\cap \mathcal C_{m-1},$ where $L_{f}^{m}$ denotes the $m^{th}$ Lie derivative along $f$ and $L_{g}$ denotes the matrix of Lie derivatives along the columns of $g$; 
$O(\cdot)=\sum_{i=1}^{m-1}L_{f}^{i}(\alpha_{m-1}\circ\psi_{m-i-1})(\boldsymbol{x})$ contains the remaining Lie derivatives along $f$ with degree less than or equal to $m-1$. $\psi_{i}(\boldsymbol{x})\ge0$ is referred to as the $i^{th}$ order HOCBF inequality (constraint in optimization). We assume that $L_{g}L_{f}^{m-1}b(\boldsymbol{x})\boldsymbol{u}\ne0$ on the boundary of set $\mathcal C_{0}\cap\ldots\cap \mathcal C_{m-1}.$ 
\end{definition}

\begin{theorem}[Safety Guarantee~\cite{xiao2021high}]
\label{thm:safety-guarantee}
Given a HOCBF $b(\boldsymbol{x})$ from Def. \ref{def:HOCBF} with corresponding sets $\mathcal{C}_{0}, \dots,\mathcal {C}_{m-1}$ defined by \eqref{eq:sequence-set1}, if $\boldsymbol{x}(0) \in \mathcal {C}_{0}\cap \dots \cap \mathcal {C}_{m-1},$ then any Lipschitz controller $\boldsymbol{u}$ that satisfies the inequality in \eqref{eq:highest-HOCBF}, $\forall t\ge 0$ renders $\mathcal {C}_{0}\cap \dots \cap \mathcal {C}_{m-1}$ forward invariant for system \eqref{eq:affine-control-system}, $i.e., \boldsymbol{x} \in \mathcal {C}_{0}\cap \dots \cap \mathcal {C}_{m-1}, \forall t\ge 0.$
\end{theorem}
\subsection{Auxiliary-Variable Adaptive Control Barrier Functions (AVCBFs)}
\label{subsec:pre-AVCBFs}
Central to AVCBFs is the introduction of several time-varying auxiliary variables that multiply corresponding CBFs, and defining dynamics for the
auxiliary variables to adapt them in constructing the corresponding constraints.
{\small Define} $m$ time-varying auxiliary variables $a_{1}(t),\dots,a_{m}(t)$ and $m$ auxiliary systems in the form 
\begin{equation}
\label{eq:virtual-system}
\dot{\boldsymbol{\pi}}_{i}=F_{i}(\boldsymbol{\pi}_{i})+G_{i}(\boldsymbol{\pi}_{i})\nu_{i}, i \in \{1,...,m\},
\end{equation}
where $\boldsymbol{\pi}_{i}(t)\coloneqq [\pi_{i,1}(t),\dots,\pi_{i,m+1-i}(t)]^{T}\in \mathbb{R}^{m+1-i}$ denotes an auxiliary state with $\pi_{i,j}(t)\in \mathbb{R}, j \in \{1,...,m+1-i\}.$ $\nu_{i}\in \mathbb{R}$ denotes an auxiliary input for \eqref{eq:virtual-system}, $F_{i}:\mathbb{R}^{m+1-i}\to\mathbb{R}^{m+1-i}$ and $G_{i}:\mathbb{R}^{m+1-i}\to\mathbb{R}^{m+1-i}$ are locally Lipschitz. One simple way to build up connection between auxiliary variables and auxiliary systems is to make $a_{i}(t)=\pi_{i,1}(t), \dot{\pi}_{i,1}(t)=\pi_{i,2}(t),\dots,\dot{\pi}_{i,m-i}(t)=\pi_{i,m+1-i}(t)$ and $\dot{\pi}_{i,m+1-i}(t)=\nu_{i}.$ Given a HOCBF $a_{i}:\mathbb{R} \to \mathbb{R}$  with relative degree $m+1-i$ for \eqref{eq:virtual-system}, we can define a sequence of functions $\varphi_{i,j}:\mathbb{R}^{m+1-i}\to\mathbb{R}, i \in\{1,...,m\}, j \in\{1,...,m+1-i\}:$
\begin{equation}
\label{eq:virtual-HOCBFs}
\varphi_{i,j}(\boldsymbol{\pi}_{i})\coloneqq\dot{\varphi}_{i,j-1}(\boldsymbol{\pi}_{i})+\alpha_{i,j}(\varphi_{i,j-1}(\boldsymbol{\pi}_{i})),
\end{equation}
where $\varphi_{i,0}(\boldsymbol{\pi}_{i})\coloneqq a_{i}(t),$ $\alpha_{i,j}(\cdot)$ are $(m+1-i-j)^{th}$ order differentiable class $\kappa$ functions. Each $a_{i}(t)$ can be controlled to be positive
according to the method in \cite{liu2023auxiliary} stemming from Theorem \ref{thm:safety-guarantee}.

Let $\boldsymbol{\Pi}(t)\coloneqq [\boldsymbol{\pi}_{1}(t),\dots,\boldsymbol{\pi}_{m}(t)]^{T}$ and $\boldsymbol{\nu}\coloneqq [\nu_{1},\dots,\nu_{m}]^{T}$ denote the auxiliary states and control inputs of system \eqref{eq:virtual-system}. Consider a sequence of functions 
\begin{small}
\begin{equation}
\label{eq:AVBCBF-sequence}
\begin{split}
&\psi_{0}(\boldsymbol{x},\boldsymbol{\Pi}(t))\coloneqq a_{1}(t)b(\boldsymbol{x}),\\
&\psi_{i}(\boldsymbol{x},\boldsymbol{\Pi}(t))\coloneqq a_{i+1}(t)(\dot{\psi}_{i-1}(\boldsymbol{x},\boldsymbol{\Pi}(t))+\alpha_{i}(\psi_{i-1}(\boldsymbol{x},\boldsymbol{\Pi}(t)))),
\end{split}
\end{equation}
\end{small}
where $i \in \{1,...,m-1\}, \psi_{m}(\boldsymbol{x},\boldsymbol{\Pi}(t))\coloneqq \dot{\psi}_{m-1}(\boldsymbol{x},\boldsymbol{\Pi}(t))+\alpha_{m}(\psi_{m-1}(\boldsymbol{x},\boldsymbol{\Pi}(t))).$ Define a sequence of sets $\mathcal{B}_{i}$ associated with \eqref{eq:AVBCBF-sequence} in the form 
\begin{small}
\begin{equation}
\label{eq:AVBCBF-set}
\begin{split}
\mathcal B_{i}\coloneqq \{(\boldsymbol{x},\boldsymbol{\Pi}(t)) \in \mathbb{R}^{n} \times \mathbb{R}^{m}:\psi_{i}(\boldsymbol{x},\boldsymbol{\Pi}(t))\ge 0\}, 
\end{split}
\end{equation}
\end{small}
where $i \in \{0,...,m-1\}$. A constraint set $\mathcal{U}_{\boldsymbol{a}}$ for $\boldsymbol{\nu}$ can be defined as 
\begin{small}
\begin{equation}
\label{eq:constraint-up}
\begin{split}
\mathcal{U}_{\boldsymbol{a}}(\boldsymbol{\Pi})\coloneqq \{\boldsymbol{\nu}\in\mathbb{R}^{m}:   L_{F_{i}}^{m+1-i}a_{i}+[L_{G_{i}}L_{F_{i}}^{m-i}a_{i}]\nu_{i}\\
+O_{i}(a_{i})+ \alpha_{i,m+1-i}(\varphi_{i,m-i}(a_{i})) \ge\epsilon, i\in \{1,\dots,m\}\},
\end{split}
\end{equation}
\end{small}
where $\varphi_{i,m-i}(\cdot)$ is similar to \eqref{eq:virtual-HOCBFs}. $O_{i}(\cdot)=\sum_{j=1}^{m-i}L_{F_{i}}^{j}(\alpha_{i,m-i}\circ\varphi_{i,m-1-i})(t) $ where $\circ$ denotes the composition of functions. $\epsilon$ is a positive constant which can be infinitely small (see \cite{liu2023auxiliary}).
\begin{definition}[AVCBF~\cite{liu2023auxiliary}]
\label{def:AVBCBF}
Let $\psi_{i}(\boldsymbol{x},\boldsymbol{\Pi}(t)),\ i\in \{1,...,m\}$ be defined by \eqref{eq:AVBCBF-sequence} and $\mathcal B_{i},\ i\in \{0,...,m-1\}$ be defined by \eqref{eq:AVBCBF-set}. A function $b(\boldsymbol{x}):\mathbb{R}^{n}\to\mathbb{R}$ is an Auxiliary-Variable Adaptive Control Barrier Function (AVCBF) with relative degree $m$ for system \eqref{eq:affine-control-system} if every $a_{i}(t),i\in \{1,...,m\}$ is a
HOCBF with relative degree $m+1-i$ for the auxiliary system
\eqref{eq:virtual-system}, and there exist $(m-j)^{th}$ order differentiable class $\kappa$ functions $\alpha_{j},j\in \{1,...,m-1\}$
and a class $\kappa$ functions $\alpha_{m}$ s.t.
\begin{small}
\begin{equation}
\label{eq:highest-AVBCBF}
\begin{split}
\sup_{\boldsymbol{u}\in \mathcal{U},\boldsymbol{\nu}\in \mathcal{U}_{\boldsymbol{a}}}[\sum_{j=2}^{m-1}[(\prod_{k=j+1}^{m}a_{k})\frac{\psi_{j-1}}{a_{j}}\nu_{j}] + \frac{\psi_{m-1}}{a_{m}}\nu_{m} \\ +(\prod_{i=2}^{m}a_{i})b(\boldsymbol{x})\nu_{1} +(\prod_{i=1}^{m}a_{i})(L_{f}^{m}b(\boldsymbol{x})+L_{g}L_{f}^{m-1}b(\boldsymbol{x})\boldsymbol{u})\\+R(b(\boldsymbol{x}),\boldsymbol{\Pi})
+ \alpha_{m}(\psi_{m-1})] \ge 0,
\end{split}
\end{equation}
\end{small}
$\forall (\boldsymbol{x},\boldsymbol{\Pi})\in \mathcal B_{0}\cap,...,\cap \mathcal B_{m-1}$ and each $a_{i}>0, i\in\{1,\dots,m\}.$ In \eqref{eq:highest-AVBCBF}, $R(b(\boldsymbol{x}),\boldsymbol{\Pi})$ denotes the remaining Lie derivative terms of $b(\boldsymbol{x})$ (or $\boldsymbol{\Pi}$) along $f$ (or $F_{i},i\in\{1,\dots,m\}$) with degree less than $m$ (or $m+1-i$), which is similar to the form of $O(\cdot )$ in \eqref{eq:constraint-up}.
\end{definition}

\begin{theorem}[Safety Guarantee with Auxiliary Variables~\cite{liu2023auxiliary}]
\label{thm:safety-guarantee-3}
Given an AVCBF $b(\boldsymbol{x})$ from Def. \ref{def:AVBCBF} with corresponding sets $\mathcal{B}_{0}, \dots,\mathcal {B}_{m-1}$ defined by \eqref{eq:AVBCBF-set}, if $(\boldsymbol{x}(0),\boldsymbol{\Pi}(0)) \in \mathcal {B}_{0}\cap \dots \cap \mathcal {B}_{m-1},$ then if there exists solution of Lipschitz controller $(\boldsymbol{u},\boldsymbol{\nu})$ that satisfies the constraint in \eqref{eq:highest-AVBCBF} and also ensures $(\boldsymbol{x},\boldsymbol{\Pi})\in \mathcal {B}_{m-1}$ for all $t\ge 0,$ then $\mathcal {B}_{0}\cap \dots \cap \mathcal {B}_{m-1}$ will be rendered forward invariant for system \eqref{eq:affine-control-system}, $i.e., (\boldsymbol{x},\boldsymbol{\Pi}) \in \mathcal {B}_{0}\cap \dots \cap \mathcal {B}_{m-1}, \forall t\ge 0.$ Moreover, $b(\boldsymbol{x})\ge 0$ is ensured for all $t\ge 0.$
\end{theorem}
Note that if auxiliary variables $a_{1}(t)=a_{2}(t)=\dots=a_{m}(t)=1,$ the constraint \eqref{eq:highest-AVBCBF} and sets \eqref{eq:AVBCBF-set} are the same as the constraint \eqref{eq:highest-HOCBF} and sets \eqref{eq:sequence-set1} respectively, which means HOCBF is a special version of AVCBF.

Several studies \cite{nguyen2016exponential},\cite{xiao2021high} integrate HOCBFs \eqref{eq:highest-HOCBF} with quadratic costs in systems with high relative degrees, creating safety-critical optimization problems that ensure safety by maintaining the forward invariance of safety-related sets. Control Lyapunov Functions (CLFs) are added to achieve exponential state convergence (see \cite{xiao2021high},\cite{xiao2021adaptive}), while approaches similar to CBF-CLF can secure finite-time convergence (see \cite{lindemann2018control,xiao2021high2,garg2019control}). In these studies, control inputs serve as decision variables in an optimization problem that is solved over discretized time intervals. Constraints from CBFs and CLFs are included, fixing the state value at each interval's start, thus creating a QP problem. The optimal control, determined by solving the QP, is applied consistently throughout the interval, with state updates based on dynamics \eqref{eq:affine-control-system}. This method, which throughout this paper we will referred to as the CBF-CLF-QP method, operates under the assumption that it is feasible to solve the QP problem at each time interval.


\section{Problem Formulation and Approach}
\label{sec:Problem Formulation and Approach}


Our objective is to generate a control strategy for system \eqref{eq:affine-control-system}, aiming for reachability to a target state within a user-defined time $T$. This strategy also seeks to minimize energy consumption, ensure compliance with safety requirement, and adhere to constraints on control inputs \eqref{eq:control-constraint}.

\textbf{Finite-Time Reachability Requirement:} The states of system \eqref{eq:affine-control-system} should reach a closed set given as $S \coloneqq \{\boldsymbol{x}\in\mathbb{R}^{n}:h(\boldsymbol{x})\le 0\}$ for a given function $h:\mathbb{R}^{n}\to\mathbb{R}$ in a user-defined interval $T,$ i.e., given $\boldsymbol{x}(0)\in\mathbb{R}^{n}, \exists t_{r} \in [0, T] \Longrightarrow \boldsymbol{x}(t_{r})\in S.$ We assume that given initial state $\boldsymbol{x}(0),$ there always exist closed-loop trajectories $\boldsymbol{x}(t)$ to satisfy the finite-time reachability requirement.


\textbf{Safety Requirement:} System \eqref{eq:affine-control-system} should always satisfy a safety requirement of the form: 
\begin{equation}
\label{eq:Safety constraint}
b(\boldsymbol{x}(t))\ge 0, \boldsymbol{x} \in \mathbb{R}^{n}, \forall t \in [0, t_{r}],
\end{equation}
where $b:\mathbb{R}^{n}\to\mathbb{R}$ is assumed to be a continuously differentiable function.

\textbf{Control Limitation Requirement:} The controller $\boldsymbol{u}$ should always satisfy \eqref{eq:control-constraint} for all $t \in [0, t_{r}].$

\textbf{Objective:} We consider the cost  
\begin{equation}
\label{eq:cost-function-1}
\begin{split}
 J(\boldsymbol{u}(t))=\int_{0}^{t_{r}} 
 D(\left \| \boldsymbol{u}(t) \right \|)dt,
\end{split}
\end{equation}
where $\left \| \cdot \right \|$ denotes the 2-norm of a vector, $D(\cdot)$ is a strictly increasing function of its argument.

A control policy is \textbf{feasible} if all constraints derived from previously mentioned requirements are satisfied and mutually non-conflicting during period $[0, t_{r}].$ In this paper, we consider the following problem:

\begin{problem}
\label{prob:SACC-prob}
Find a feasible control policy for system \eqref{eq:affine-control-system} such that the previously mentioned requirements are satisfied and cost \eqref{eq:cost-function-1} is minimized.
\end{problem}

As mentioned at the end of Sec. \ref{sec:Preliminaries}, several approaches exist for the case when the cost in (\ref{eq:cost-function-1}) 
is quadratic and the required convergence is in infinite time. In most of these works, the feasibility of solving the QP in every interval cannot be assured and is, in reality, improbable, especially when the control bounds in Eqn. \eqref{eq:control-constraint} are tight. Moreover, achieving finite-time reachability is significantly more complex than merely ensuring the eventual reachability of states. This complexity implies that there will be fewer solutions available in the whole state space. Given that the solutions to the CBF-CLF-QP method reside within this state space, solving the QP problem becomes more prone to infeasibility.

To address these limitations, our approach to Problem \ref{prob:SACC-prob} proceeds as follows. We define a HOCBF to enforce \eqref{eq:Safety constraint}. Since the constraints related to AVCBFs are relaxed by auxiliary
inputs and this relaxation still ensures the forward invariance
of sets, which increases the overall feasibility of solving QPs, we generalize AVCBFs to Auxiliary-Variable Adaptive Control Lyapunov Barrier Functions (AVCLBFs) in Sec. \ref{sec:AVCLFBs}. Specifically, we design an AVCLBF and associated sets to achieve the finite-time reachability requirement. If the forward invariance of sets associated with AVCLBF can be guaranteed, then the finite-time reachability requirement is assured. Consequently, we identify a feasible optimal control that does not conflict with safety and input constraints.
 
\section{Auxiliary-Variable Adaptive Control Lyapunov Barrier Functions}
\label{sec:AVCLFBs}

In this section, we begin with a simple example that illustrates our proposed approach. 

\subsection{Motivating Example}
Consider a simplified unicycle model expressed as 
\begin{small}
\begin{equation}
\label{eq:UM-dynamics}
\underbrace{\begin{bmatrix}
\dot{x}(t) \\
\dot{y}(t) \\
\dot{\theta}(t)
\end{bmatrix}}_{\dot{\boldsymbol{x}}(t)}  
=\underbrace{\begin{bmatrix}
 v\cos{(\theta(t))}  \\
 v\sin{(\theta(t))} \\
 0
\end{bmatrix}}_{f(\boldsymbol{x}(t))} 
+ \underbrace{\begin{bmatrix}
  0 \\
  0 \\
  1 
\end{bmatrix}}_{g(\boldsymbol{x}(t))}u(t),
\end{equation}
\end{small}
where $(x, y)$ denote the coordinates of the unicycle, $v>0$ is its linear speed (assumed constant), $\theta$ denotes the heading angle, and $u$ represents the angular speed. We require the unicycle to visit a disk at some finite time $t_{r}$, i.e., $\exists t_{r} \in [0, T]$ such that $(x(t_{r})-x_{0})^{2}+(y(t_{r})-y_{0})^{2}-r^2\le 0$, where $(x_{0}, y_{0})$ and $r$ denote the center location and radius of the disk, respectively. We can define a CLF like function $V(\boldsymbol{x})\coloneqq (x-x_{0})^{2}+(y-y_{0})^{2}-r^2$, and then enforce the derived inequality
\begin{small}
\begin{equation}
\label{eq:CLF constraint}
\psi_{1}(\boldsymbol{x})\coloneqq \dot{V}(\boldsymbol{x})+c_{1}V(\boldsymbol{x})^{\frac{1}{3}}\le0,
\end{equation}
\end{small}
where $c_{1}>0$ determines the convergence rate. From \eqref{eq:CLF constraint}, we have
\begin{small}
\begin{equation}
\label{eq:CLF constraint2}
V(\boldsymbol{x}(t))\le \left [V(\boldsymbol{x}(0))^{\frac{2}{3}} -\frac{2}{3}c_{1}t \right ]^{\frac{3}{2}}, \text{if } V(\boldsymbol{x}(t))\ge0.
\end{equation}
\end{small}
Assume $T=10s$. If $ c_{1}=\frac{3V(\boldsymbol{x}(0))^{\frac{2}{3}}}{20},$ based on \eqref{eq:CLF constraint2} we have $V(\boldsymbol{x}(10))\le 0,$ which means the value of $V(\boldsymbol{x})$ will converge to 0 at or before 10 seconds. If $ c_{1}>\frac{3V(\boldsymbol{x}(0))^{\frac{2}{3}}}{20},$ the value of $V(\boldsymbol{x})$ will converge to 0 before 10 seconds. However, since the function of $V(\boldsymbol{x})$ has relative degree 2 for system \eqref{eq:UM-dynamics} and the control input does not appear in \eqref{eq:CLF constraint}, we cannot directly apply inequality \eqref{eq:CLF constraint} as a constraint under QPs.
The authors of \cite{xiao2021high2} proposed High Order Control Lyapunov Barrier Functions (HOCLBFs) which can be used to extend $V(\boldsymbol{x})$ into higher order function, e.g., $\psi_{2}(\boldsymbol{x},\boldsymbol{u})\coloneqq \dot{\psi}_{1}(\boldsymbol{x})+c_{2}\psi_{1}(\boldsymbol{x})^{\frac{1}{3}}.$ One issue with this method is that the value of $\psi_{1}(\boldsymbol{x}(0))$ might not be negative initially. If $c_{2}=\frac{3\psi_{1}(\boldsymbol{x}(0))^{\frac{2}{3}}}{2t_{2}},$ it will take at most $t=t_{2}$ to have $\psi_{1}(\boldsymbol{x})$ converge to 0, then at most $t=t_{1}$ to have $V(\boldsymbol{x})$ converge to 0, given $c_{1}=\frac{3V(\boldsymbol{x}(t_{2}))^{\frac{2}{3}}}{2t_{1}}.$ The parameters related to time should be selected first to satisfy $t_{1}+t_{2} \le 10s,$ then based on $t_{1},t_{2}$, parameters related to convergence rate should be reversely selected, i.e., determine $c_{2}$ first, then after $t_{2}$ determine $c_{1}.$ This leads to some parameters not being determined initially (like $c_{1}$ in above case), thereby posing difficulties for global planning. Another issue is the parameters $c_{1},c_{2}$ are not time-varying, which will significantly reduce the number of solutions available in the whole state space, thus resulting in frequent infeasibility. To address these issues, we introduce AVCLBFs next.
\subsection{Auxiliary-Variable Adaptive Control Lyapunov Barrier Functions}
For finite-time reachability requirement, consider a function $h(\boldsymbol{x})$ in the form
\begin{equation}
\label{eq:Convergence hx}
h(\boldsymbol{x})\coloneqq \left \| \boldsymbol{x}(t)-\boldsymbol{x}_{e} \right \|_{p}-\delta_{e},
\end{equation}
where $\left \| \cdot \right \|_{p}$ denotes the $p$-norm of a vector, $\boldsymbol{x}_{e} \in \mathbb{R}^{n}$ is a desired state vector, and $\delta_{e}>0$ denotes permissible reachability tolerance. We assume the initial value of $h(\boldsymbol{x})$ is positive as $h(\boldsymbol{x}(0))>0$ and the relative degree of $h(\boldsymbol{x})$ with respect to \eqref{eq:affine-control-system} is $m, m\ge 2.$ Motivated by \eqref{eq:virtual-system} and \eqref{eq:AVBCBF-sequence}, we first define $m-1$ time-varying auxiliary variables $a_{1}(t),\dots,a_{m-1}(t)$ and $m-1$ auxiliary systems, then define a sequence of functions $\varphi_{i,j}:\mathbb{R}^{m-i}\to\mathbb{R}, i \in\{1,...,m-1\}, j \in\{1,...,m-i\}$ similar to \eqref{eq:virtual-HOCBFs} to ensure each $a_{i}(t)$ is positive. A function $\psi_{0}$ is then defined as 
\begin{small}
\begin{equation}
\label{eq:original CLF}
\psi_{0}(\boldsymbol{x},\boldsymbol{\Pi}(t))\coloneqq a_{1}(t)(-\dot{h}(\boldsymbol{x})-c(t)\beta(h(\boldsymbol{x}))),
\end{equation}
\end{small}
where $c(t)$ is a time-varying variable and $\beta(\cdot)$ is an extended class $\kappa$ function. We notice the number of auxiliary variables $a_{i}(t)$ here is $m-1$, one less compared to $m$ in Sec. \ref{subsec:pre-AVCBFs} since the relative degree of $\psi_{0}(\boldsymbol{x},\boldsymbol{\Pi}(t))$ here with respect to \eqref{eq:affine-control-system} is $m-1.$ We define a sequence of functions $\psi_{i}:\mathbb{R}^{n}\to\mathbb{R},\ i\in \{1,...,m-2\}$ as
\begin{small}
\begin{equation}
\label{eq:sequence-f1-2}
\begin{split}
\psi_{i}(\boldsymbol{x},\boldsymbol{\Pi}(t))\coloneqq a_{i+1}(t)(\dot{\psi}_{i-1}(\boldsymbol{x},\boldsymbol{\Pi}(t))+\alpha_{i}(\psi_{i-1}(\boldsymbol{x},\boldsymbol{\Pi}(t)))), \\
\psi_{m-1}(\boldsymbol{x},\boldsymbol{\Pi}(t))\coloneqq \dot{\psi}_{m-2}(\boldsymbol{x},\boldsymbol{\Pi}(t))+\alpha_{m-1}(\psi_{m-2}(\boldsymbol{x},\boldsymbol{\Pi}(t))),
\end{split}
\end{equation}
\end{small}
where $i \in \{1,...,m-2\},$ $\boldsymbol{\Pi}(t)\coloneqq [\boldsymbol{\pi}_{1}(t),\dots,\boldsymbol{\pi}_{m-1}(t)]^{T}$ and $\boldsymbol{\nu}\coloneqq [\nu_{1},\dots,\nu_{m-1}]^{T}$ denote the auxiliary states and control inputs of $m-1$ auxiliary systems defined in \eqref{eq:virtual-system}. A sequence of sets $\mathcal{B}_{i}$ ($i \in \{0,...,m-2\}$) associated with \eqref{eq:sequence-f1-2} and constraint set $\mathcal{U}_{\boldsymbol{a}}$ for $\boldsymbol{\nu}$ are defined corresponding to \eqref{eq:AVBCBF-set} and \eqref{eq:constraint-up} respectively.
\begin{definition}[Auxiliary-Variable Adaptive Control Lyapunov Barrier Function (AVCLBF)]
\label{def:AVCLBF}
Let $\psi_{i}(\boldsymbol{x},\boldsymbol{\Pi}(t)),\ i\in \{0,...,m-1\}$ be defined by  \eqref{eq:original CLF}, \eqref{eq:sequence-f1-2} and $\mathcal B_{i},\ i\in \{0,...,m-2\}$ be defined by \eqref{eq:AVBCBF-set}. A function $h(\boldsymbol{x}):\mathbb{R}^{n}\to\mathbb{R}$ is an Auxiliary-Variable Adaptive Control Lyapunov Barrier Function (AVCLBF) with relative degree $m$ ($m\ge2$) for system \eqref{eq:affine-control-system} if every $a_{i}(t),i\in \{1,...,m-1\}$ is a
HOCBF with relative degree $m-i$ for the auxiliary system
\eqref{eq:virtual-system}, and there exist $(m-1-j)^{th}$ order differentiable class $\kappa$ functions $\alpha_{j},j\in \{1,...,m-2\}$
and a class $\kappa$ function $\alpha_{m-1}$ s.t.
\begin{small}
\begin{equation}
\label{eq:highest-AVCLBF}
\begin{split}
\sup_{\boldsymbol{u}\in \mathcal{U},\boldsymbol{\nu}\in \mathcal{U}_{\boldsymbol{a}}}[\sum_{j=2}^{m-2}[(\prod_{k=j+1}^{m-1}a_{k})\frac{\psi_{j-1}}{a_{j}}\nu_{j}] + \frac{\psi_{m-2}}{a_{m-1}}\nu_{m-1} \\ +(\prod_{i=2}^{m-1}a_{i})\psi_{0}\nu_{1} +(\prod_{i=1}^{m-1}a_{i})(L_{f}^{m-1}\psi_{0}+L_{g}L_{f}^{m-2}\psi_{0}\boldsymbol{u})\\+R(\psi_{0},\boldsymbol{\Pi})
+ \alpha_{m-1}(\psi_{m-2})] \ge 0,
\end{split}
\end{equation}
\end{small}
for all $(\boldsymbol{x},\boldsymbol{\Pi})\in \mathcal B_{0}\cap,...,\cap \mathcal B_{m-2}$ and each $a_{i}>0, i\in\{1,\dots,m-1\}.$ In \eqref{eq:highest-AVCLBF}, $R(\psi_{0},\boldsymbol{\Pi})$ denotes the remaining Lie derivative terms of $\psi_{0}$ (or $\boldsymbol{\Pi}$) along $f$ (or $F_{i},i\in\{1,\dots,m-1\}$) with degree less than $m-1$ (or $m-i$), which is similar to the form of $O(\cdot )$ in \eqref{eq:constraint-up}.
\end{definition}

\begin{theorem}
\label{thm:safety-guarantee-4}
Given an AVCLBF $h(\boldsymbol{x})$ from Def. \ref{def:AVCLBF} with corresponding sets $\mathcal{B}_{0}, \dots,\mathcal {B}_{m-2}$ defined by \eqref{eq:AVBCBF-set}, if $(\boldsymbol{x}(0),\boldsymbol{\Pi}(0)) \in \mathcal {B}_{0}\cap \dots \cap \mathcal {B}_{m-2},$ then if there exists a Lipschitz controller $(\boldsymbol{u},\boldsymbol{\nu})$ that satisfies the constraint in \eqref{eq:highest-AVCLBF} and also ensures $(\boldsymbol{x},\boldsymbol{\Pi})\in \mathcal {B}_{m-2}$ for all $t\ge 0,$ then $\mathcal {B}_{0}\cap \dots \cap \mathcal {B}_{m-2}$ will be rendered forward invariant for system \eqref{eq:affine-control-system}, $i.e., (\boldsymbol{x},\boldsymbol{\Pi}) \in \mathcal {B}_{0}\cap \dots \cap \mathcal {B}_{m-2}, \forall t\ge 0.$ Moreover, $\frac{\psi_{0}(\boldsymbol{x},\boldsymbol{\Pi}(t))}{a_{1}(t)}\coloneqq -\dot{h}(\boldsymbol{x})-c(t)\beta(h(\boldsymbol{x})) \ge 0$ is ensured for all $t\ge 0.$
\end{theorem}

The proof of the above theorem is similar to the proof of Thm. \ref{thm:safety-guarantee-3} (see \cite{liu2023auxiliary}). Based on Thm. \ref{thm:safety-guarantee-4}, since $\psi_{0}(\boldsymbol{x},\boldsymbol{\Pi}(t))\ge 0$ is ensured, the function regarding $\frac{\psi_{0}(\boldsymbol{x},\boldsymbol{\Pi}(t))}{a_{1}(t)}\coloneqq -\dot{h}(\boldsymbol{x})-c(t)\beta(h(\boldsymbol{x})) \ge 0$ is guaranteed.
\begin{remark}
\label{rem:m bigger than 2} 
Some comments are in order on our assumption that the relative degree of $h(\boldsymbol{x})$ for system \eqref{eq:affine-control-system} is higher than or equal to 2 in Def. \ref{def:AVCLBF}. If $m=1,$ the highest order equation becomes $\psi_{0}(\boldsymbol{x},\boldsymbol{u})\coloneqq-\dot{h}(\boldsymbol{x})-c(t)\beta(h(\boldsymbol{x})),$ which means auxiliary variable ${a_{i}(t)}$ will not be introduced and an adjustable parameter associate with $\psi_{0}(\boldsymbol{x},\boldsymbol{u})$ is $c(t).$ We can also introduce a relaxation variable $w$ in this case as $-\dot{h}(\boldsymbol{x})-c(t)\beta(h(\boldsymbol{x}))\ge w$ to increases the feasibility of solving QPs.
\end{remark}
We consider the power function as a general form
of extended class $\kappa$ function as $\beta(h(\boldsymbol{x}))\coloneqq h(\boldsymbol{x})^{q}, q>0.$ Based on this, we have:
\begin{lemma}
\label{lemma:state convergence}
Given a function $h(\boldsymbol{x}),$ if the next inequality is satisfied:
\begin{equation}
\label{eq:state convergence}
\dot{h}(\boldsymbol{x})+c(t)h(\boldsymbol{x})^{q}\le0, \forall t\ge0,
\end{equation}
with $q\in (0,1)$ and $h(\boldsymbol{x}(0))>0,$ then there exists an upper bound for $h(\boldsymbol{x}),$ and the time at which this upper bound reaches 0 is $T$, where $C_{\int}(t)=\int_0^{t} c(\tau)d\tau$ and $C_{\int}(T)=C_{\int}(0)+\frac{h(\boldsymbol{x}(0))^{1-q}}{1-q}.$
\end{lemma}
\begin{proof}
Based on $\dot{h} (\boldsymbol{x})+c(t)h(\boldsymbol{x})^{q}=0, C_{\int}(t)=\int_0^{t} c(\tau)d\tau$ and $q\in (0,1)$, we have
\begin{small}
\begin{equation}
\label{eq:state convergence time}
h(\boldsymbol{x})\coloneqq \left [h(\boldsymbol{x}(0))^{1-q}-(1-q)(C_{\int}(t)-C_{\int}(0))\right ]^{\frac{1}{1-q}},
\end{equation}
\end{small}
where $h(\boldsymbol{x})\ge 0.$ In \eqref{eq:state convergence time}, the function $h(\boldsymbol{x}(0))^{1-q}-(1-q)(C_{\int}(t)-C_{\int}(0))$ will reach 0 at time $T$ if $C_{\int}(T)=C_{\int}(0)+\frac{h(\boldsymbol{x}(0))^{1-q}}{1-q}.$ 
Based on the comparison lemma in \cite{Khalil:1173048},
since \eqref{eq:state convergence},\eqref{eq:state convergence time} hold, we have
\begin{small}
\begin{equation}
\label{eq:state convergence time2}
h(\boldsymbol{x})\le \left [h(\boldsymbol{x}(0))^{1-q}-(1-q)(C_{\int}(t)-C_{\int}(0))\right ]^{\frac{1}{1-q}},
\end{equation}
\end{small}
where $h(\boldsymbol{x})\ge 0.$ Thus, the upper bound of $h(\boldsymbol{x})$ will reach 0 at $T$ where $C_{\int}(T)=C_{\int}(0)+\frac{h(\boldsymbol{x}(0))^{1-q}}{1-q}.$ $C_{\int}(T)$ is called the critical value for finite-time reachability.
\end{proof}

\begin{remark}
\label{rem:convergence to 0 guarantee} 
Note that Eqns. \eqref{eq:state convergence time}, \eqref{eq:state convergence time2} are correct only when $h(\boldsymbol{x})\ge 0$ is satisfied. Since the upper bound of $h(\boldsymbol{x})$ will reach 0 at $T$ based on Lemma. \ref{lemma:state convergence}, it is possible that $h(\boldsymbol{x})$ reaches 0 then becomes negative at $t_{r}$ where $t_{r}<T.$ Typically, we consider that the requirement of finite-time reachability has been completed under such circumstances. If we want to keep $h(\boldsymbol{x})$ negative during period $t\in [t_{r},T],$ we should rewrite \eqref{eq:state convergence time2} into
\begin{equation}
\begin{small}
\label{eq:state convergence time special}
h(\boldsymbol{x})\le \begin{cases}
 \left [h(\boldsymbol{x}(t_{r}))^{1-q}-(1-q)(C_{\int}(t)-C_{\int}(t_{r}))\right ]^{\frac{1}{1-q}}, (a), \\
 -\left [h(\boldsymbol{x}(t_{r}))^{1-q}-(1-q)(C_{\int}(t)-C_{\int}(t_{r}))\right ]^{\frac{1}{1-q}}, (b),
\end{cases}
\end{small}
\end{equation}
where $h(\boldsymbol{x}(t_{r}))^{1-q}<0$ for $(a),$ $h(\boldsymbol{x}(t_{r}))^{1-q}>0$ for $(b), \forall t\in [t_{r},T].$ If $h(\boldsymbol{x}(t_{r}))^{1-q}$ is an imaginary number, reset the $q$ at $t=0$ to make $h(\boldsymbol{x}(t_{r}))^{1-q}$ a real number. As long as inequality $C_{\int}(t)\ge C_{\int}(t_{r})+\frac{h(\boldsymbol{x}(t_{r}))^{1-q}}{1-q}$ for $(a)$ or $C_{\int}(t)\le C_{\int}(t_{r})+\frac{h(\boldsymbol{x}(t_{r}))^{1-q}}{1-q}$ for $(b)$ satisfies $\forall t\in [t_{r},T],$ based on \eqref{eq:state convergence time special}, $h(\boldsymbol{x})\le0$ will satisfy during period $t\in [t_{r},T].$ We don't need to prevent chattering behavior by switching classes of CLBFs (see \cite{xiao2021high2}), we just need to adjust $C_{\int}(t)$ to meet the corresponding inequalities, which avoids extensive parameter tuning.
\end{remark}

\begin{remark}
\label{rem:range of c} 
If $q =1$ in Lemma \ref{lemma:state convergence}, based on $\dot{h} (\boldsymbol{x})+c(t)h(\boldsymbol{x})^{q}=0$ and $ C_{\int}(t)=\int c(t)dt,$ we have 
\begin{equation}
\label{eq:state convergence time wrong1}
h(\boldsymbol{x})\coloneqq h(\boldsymbol{x}(0))e^{(C_{\int}(0)-C_{\int}(t))}.
\end{equation}
If $q >1$ in lemma. \ref{lemma:state convergence}, we have 
\begin{equation}
\label{eq:state convergence time wrong2}
h(\boldsymbol{x})\coloneqq \left [h(\boldsymbol{x}(0))^{1-q}-(1-q)(C_{\int}(t)-C_{\int}(0))\right ]^{\frac{1}{1-q}},
\end{equation}
where $h(\boldsymbol{x})\ge0.$ For \eqref{eq:state convergence time wrong1},\eqref{eq:state convergence time wrong2}, As $C_{\int}(t)$ approaches $+\infty,$ $h(\boldsymbol{x})$ nears 0, meaning the upper bound of $h(\boldsymbol{x})$ will never hit 0 within finite time. Thus, using the comparison lemma in \cite{Khalil:1173048}, finite-time reachability cannot be assured. The range of $q$ is limited to $(0, 1).$
 \end{remark}

\begin{remark}
The benefit of making $c(t)$ a time-varying variable is shown in Fig. \ref{fig:flexibility of c}. Based on Lemma \ref{lemma:state convergence}, at time $t_{r}$ the upper bound of $h(\boldsymbol{x})$ reaches 0 where $t_{r}\le T, C_{\int}(t)=\int_0^{t} c(\tau)d\tau$ and $C_{\int}(t_{r})=C_{\int}(0)+\frac{h(\boldsymbol{x}(0))^{1-q}}{1-q}.$ This reveals the value of $C_{\int}(t)$ should reach $C_{\int}(0)+\frac{h(\boldsymbol{x}(0))^{1-q}}{1-q}$ no later than time $T$ in order to satisfy the finite-time reachability requirement. Moreover, $C_{\int}(0)$ can be any value. Curve $(a)$ shows linear growth of $C_{\int}(t)$, while curves $(b)$ and $(c)$ depict nonlinear growth. This variety in $C_{\int}(t)$ offers more solutions for \eqref{eq:state convergence}, enabling flexible, adaptive control strategies across the state space and enhancing the feasibility of solving QPs.
\label{rem:flexibility of c} 
\begin{figure}[ht]
    \centering
    \includegraphics[scale=0.36]{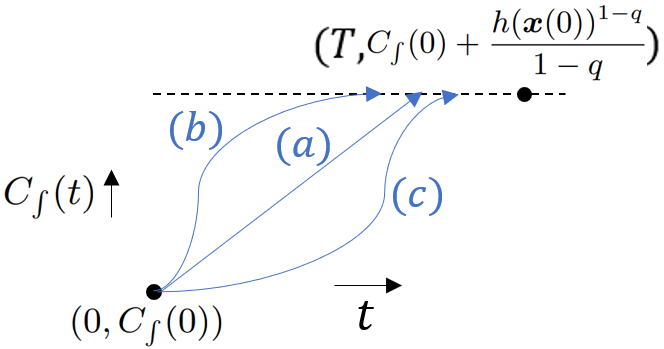}
    \caption{Schematic diagram of the changes in $C_{\int}(t)$ over time. $(0,C_{\int}(0))$ are the initial values of $t$ and $C_{\int}(t),$ respectively. $(T,C_{\int}(0)+\frac{h(\boldsymbol{x}(0))^{1-q}}{1-q})$ are the user-defined time and critical value for finite-time reachability, respectively. The three blue solid curves denote different choices regarding equation $C_{\int}(t)$ that meet the reachability requirement.}
    \label{fig:flexibility of c}
\end{figure} 
\end{remark}
\begin{theorem}
\label{thm:convergence-guarantee}
Given an AVCLBF $h(\boldsymbol{x})$ stated in Def. \ref{def:AVCLBF} and Lemma. \ref{lemma:state convergence} with $h(\boldsymbol{x}(0))>0, q\in(0,1)$, if $(\boldsymbol{x}(0),\boldsymbol{\Pi}(0)) \in \mathcal {B}_{0}\cap \dots \cap \mathcal {B}_{m-2},$ any controller $\boldsymbol{u}\in \mathcal{U}$ ($\boldsymbol{\nu}\in \mathcal{U}_{\boldsymbol{a}}$) that satisfies \eqref{eq:highest-AVCLBF} makes $h(\boldsymbol{x})$ reach 0 within time
\begin{small}
\begin{equation}
\label{eq:time upper bound}
T=C_{\int}^{-1}(C_{\int}(0)+\frac{h(\boldsymbol{x}(0))^{1-q}}{1-q}),    
\end{equation}
\end{small}
where $C_{\int}^{-1}(\cdot)$ denotes the inverse function of $C_{\int}(t).$
\end{theorem}
\begin{proof}
Since $h(\boldsymbol{x})$ is an AVCLBF stated in Def. \ref{def:AVCLBF} and Lemma. \ref{lemma:state convergence}, Eqn. \eqref{eq:original CLF} becomes
\begin{small}
\begin{equation}
\label{eq:power function}
\psi_{0}(\boldsymbol{x},\boldsymbol{\Pi}(t))\coloneqq a_{1}(t)(-\dot{h}(\boldsymbol{x})-c(t)h(\boldsymbol{x})^{q}).
\end{equation}
\end{small}
Based on Thm. \ref{thm:safety-guarantee-4}, since $\psi_{0}(\boldsymbol{x},\boldsymbol{\Pi}(t))\ge0$ is guaranteed for $t\ge0,$ we have 
\begin{small}
\begin{equation}
\label{eq:state convergence inequality}
-\frac{\psi_{0}(\boldsymbol{x},\boldsymbol{\Pi}(t))}{a_{1}(t)}=\dot{h}(\boldsymbol{x})+c(t)h(\boldsymbol{x})^{q}\le 0, \forall t\ge0.
\end{equation}
\end{small}
Based on Lemma \ref{lemma:state convergence}, the upper bound of $h(\boldsymbol{x})$ will be 0 at $T=C_{\int}^{-1}(C_{\int}(0)+\frac{h(\boldsymbol{x}(0))^{1-q}}{1-q}),$ thus the value of $h(x)$ will reach 0 before or right on $T.$
\end{proof}

\section{Case Study and Simulations}
\label{sec:Case Study and Simulations}

In this section, we consider the unicycle model with the dynamics given by \eqref{eq:UM-dynamics2} for Prob. \ref{prob:SACC-prob}, which is more realistic than the simplified unicycle model in Sec. \ref{sec:AVCLFBs} and in the case study introduced in \cite{xiao2021high2}.  
\begin{small}
\begin{equation}
\label{eq:UM-dynamics2}
\underbrace{\begin{bmatrix}
\dot{x}(t) \\
\dot{y}(t) \\
\dot{\theta}(t)\\
\dot{v}(t)
\end{bmatrix}}_{\dot{\boldsymbol{x}}(t)}  
=\underbrace{\begin{bmatrix}
 v(t)\cos{(\theta(t))}  \\
 v(t)\sin{(\theta(t))} \\
 0 \\
 0
\end{bmatrix}}_{f(\boldsymbol{x}(t))} 
+ \underbrace{\begin{bmatrix}
  0 & 0\\
  0 & 0\\
  1 & 0\\
  0 & 1
\end{bmatrix}}_{g(\boldsymbol{x}(t))}\underbrace{\begin{bmatrix}
   u_{1}(t)   \\
  u_{2}(t) 
\end{bmatrix}}_{\boldsymbol{u}(t)}
\end{equation}
\end{small}
In \eqref{eq:UM-dynamics2}, $(x, y)$ denote the coordinates of the unicycle, $v$ is its linear speed, $\theta$ denotes the heading angle, and $\boldsymbol{u}$ represent the angular velocity ($u_{1}$) and the linear acceleration ($u_{2}$), respectively. 
The objective is to minimize the control effort $\min_{\boldsymbol{u}(t)} \int_{0}^{t_{r}}  
 \boldsymbol{u}(t)^{T}\boldsymbol{u}(t)dt.$
\subsection{Safety Requirement}
For the safety requirement, we consider the case when the robot has to avoid circular obstacles. The candidate HOCBF $b(\boldsymbol{x})$ is defined based on a quadratic distance function $b_{i}(\boldsymbol{x})= (x-x_{i,o})^{2}+(y-y_{i,o})^{2}-r_{i,o}^{2}$, where $(x_{i,o},y_{i,o})$ and $r_{i,o}$ denote the $i^{th}$ obstacle center location and radius, respectively. The HOCBFs are then defined as
\begin{small}
\begin{equation}
\label{eq:HOCBF-sequence-UM2}
\begin{split}
&\psi_{i,0}(\boldsymbol{x})\coloneqq b_{i}(\boldsymbol{x}),\\
&\psi_{i,1}(\boldsymbol{x})\coloneqq L_{f}b_{i}(\boldsymbol{x})+k_{1}b_{i}(\boldsymbol{x}),\\
&\psi_{i,2}(\boldsymbol{x},\boldsymbol{u})\coloneqq L_{f}^{2}b_{i}(\boldsymbol{x})+L_{f}L_{g}b_{i}(\boldsymbol{x})\boldsymbol{u}+\\
&k_{1}L_{f}b_{i}(\boldsymbol{x})+k_{2}\psi_{i,1}(\boldsymbol{x}),
\end{split}
\end{equation}
\end{small}
where $\alpha_{i,1}(\cdot),\alpha_{i,2}(\cdot)$ are set as linear functions and $k_{1}>0, k_{2}>0.$ Note that if we want the robot to stay inside a safe circular area decided by center location $(x_{i,o},y_{i,o})$ and radius $r_{i,o}$, the candidate HOCBF should be defined as $b(\boldsymbol{x})= r_{i,o}^{2}-(x-x_{i,o})^{2}-(y-y_{i,o})^{2}.$ 

\subsection{Finite-Time Reachability Requirement}
For finite-time reachability requirement, we consider a unicycle robot tries to reach circular areas within a desired time $T$. The candidate AVCLBF $h(\boldsymbol{x})$ is defined based on a quadratic distance function $h_{i}(\boldsymbol{x})= (x-x_{i,d})^{2}+(y-y_{i,d})^{2}-r_{i,d}^{2}$, where $(x_{i,d},y_{i,d})$ and $r_{i,d}$ denote the $i^{th}$ area center location and radius, respectively. Motivated by Sec. \ref{sec:AVCLFBs}, we define an auxiliary dynamic as
\begin{small}
\begin{equation}
\label{eq:Auxiliary-dynamics1}
\underbrace{
\dot{a_{1}}(t)
}_{\dot{\boldsymbol{\pi}}_{1}(t)}  
=\underbrace{
 0
}_{F_{1}(\boldsymbol{{\pi}}_{1}(t))} 
+ \underbrace{
  1 
}_{G_{1}(\boldsymbol{{\pi}}_{1}(t))}\nu_{1}(t).
\end{equation}
\end{small}
The HOCBFs for $a_{1}(t)$ are defined as
\begin{small}
\begin{equation}
\label{eq:SHOCBF-sequence-UM2}
\begin{split}
&\varphi_{0}(\boldsymbol{{\pi}}_{1}(t))\coloneqq a_{1}(t),\\
&\varphi_{1}(\boldsymbol{{\pi}}_{1}(t))\coloneqq \dot{\varphi}_{0}(\boldsymbol{{\pi}}_{1}(t))+l_{1}\varphi_{0}(\boldsymbol{{\pi}}_{1}(t)),
\end{split}
\end{equation}
\end{small}
where $\alpha_{1}(\cdot)$ is set as a linear function and $l_{1}>0.$ Note that to ensure $a_{1}(t)>0,$ the inequality $\varphi_{1}(\boldsymbol{{\pi}}_{1}(t))\ge\epsilon$ must be satisfied (see \eqref{eq:constraint-up}). The AVCLBFs are then defined as
\begin{small}
\begin{equation}
\label{eq:AVCLBF-sequence-UM2}
\begin{split}
&\psi_{i,0}(\boldsymbol{x},\boldsymbol{{\pi}}_{1}(t))\coloneqq a_{1}(t)(-\dot{h}_{i}(\boldsymbol{x})-c(t)h_{i}(\boldsymbol{x})^{q}),\\
&\psi_{i,1}(\boldsymbol{x},\boldsymbol{{\pi}}_{1}(t))\coloneqq \dot{\psi}_{i,0}(\boldsymbol{x},\boldsymbol{{\pi}}_{1}(t))+l_{2}\psi_{i,0}(\boldsymbol{x},\boldsymbol{{\pi}}_{1}(t)),
\end{split}
\end{equation}
\end{small}
where $\alpha_{i,1}(\cdot)$ and $\beta_{i}(\cdot)$ are set as a linear function and a power function respectively with $l_{2}>0, q\in(0,1).$
\subsection{Complete Cost Function for AVCLBF-HOCBF-QP}
By formulating the constraints from HOCBFs \eqref{eq:HOCBF-sequence-UM2},\eqref{eq:SHOCBF-sequence-UM2}, AVCLBFs \eqref{eq:AVCLBF-sequence-UM2}, and control limitations \eqref{eq:control-constraint}, we can define the cost function 
 for QP as
\begin{small}
\begin{equation}
\label{eq:AVCLBF-HOCBF-QP}
\begin{split}
\min_{\boldsymbol{u}(t),\nu_{1}(t)} \int_{0}^{T}[\boldsymbol{u}(t)^{T}\boldsymbol{u}(t)+W_{1}(\nu_{1}(t)-a_{1,w})^{2}]dt,
\end{split}
\end{equation}
\end{small}
where $W_{1}$ is a positive scalar and $a_{1,w}\in \mathbb{R}$ is the scalar to which we hope the auxiliary input $\nu_{1}$ converges.
\subsection{Benchmarks for Finite-Time Reachability}
To showcase AVCLBFs' advantage in achieving finite-time reachability, we compare its performance against two benchmarks, including HOCLBFs as proposed in \cite{xiao2021high2}. To align with HOCLBFs' format, we define:
\begin{small}
\begin{equation}
\label{eq:HOCLBF-sequence-UM2}
\begin{split}
&h_{b1}(\boldsymbol{x})= r_{d}^{2}-(x-x_{d})^{2}-(y-y_{d})^{2},\\
&\psi_{1}(\boldsymbol{x})\coloneqq \dot{h}_{b1}(\boldsymbol{x})+c_{1}h_{b1}(\boldsymbol{x})^{q_{1}},\\
&\psi_{2}(\boldsymbol{x})\coloneqq \dot{\psi}_{1}(\boldsymbol{x})+c_{2}\psi_{1}(\boldsymbol{x})^{q_{2}},
\end{split}
\end{equation}
\end{small}
where $\alpha_{1}(\cdot),\alpha_{2}(\cdot)$ are set as power functions and $c_{1}>0, c_{2}>0, q_{1}\in(0,1),q_{2}\in(0,1).$ 
Another benchmark is based on time-varying CBFs proposed in \cite{lindemann2018control}. To match the format of time-varying CBFs, we define:
\begin{small}
\begin{equation}
\label{eq:TVCBF-sequence-UM2}
\begin{split}
&h_{b2}(\boldsymbol{x})= r_{d_{0}}^{2}-\frac{r_{d_{0}}^{2}-r_{d_{T}}^{2}}{T}t- (x-x_{d})^{2}-(y-y_{d})^{2},\\
&\psi_{1}(\boldsymbol{x})\coloneqq \dot{h}_{b1}(\boldsymbol{x})+l_{1}h_{b1}(\boldsymbol{x}),\\
&\psi_{2}(\boldsymbol{x})\coloneqq \dot{\psi}_{1}(\boldsymbol{x})+l_{2}\psi_{1}(\boldsymbol{x}),
\end{split}
\end{equation}
\end{small}
where $\alpha_{1}(\cdot),\alpha_{2}(\cdot)$ are set as linear functions and $l_{1}>0, l_{2}>0.$ $r_{d_{0}}$ and $r_{d_{T}}$ denote the original radius and final radius respectively of the circle area in which we hope the robot to stay. The time-varying equation $r(t)=\sqrt{r_{d_{0}}^{2}-\frac{r_{d_{0}}^{2}-r_{d_{T}}^{2}}{T}t}$ denotes the radius of a contracting circle that decreases over time, guiding the robot into a smaller circular space until it reaches the desired area within time $T.$ The cost function for benchmark methods is $\min_{\boldsymbol{u}(t)} \int_{0}^{t_{r}} 
 \boldsymbol{u}(t)^{T}\boldsymbol{u}(t)dt.$

\subsection{Simulation Results}

In this subsection, we show how our AVCLBF method outperforms the time-varying CBF and HOCLBF in solving the Prob. \ref{prob:SACC-prob} with model \eqref{eq:UM-dynamics2} in MATLAB. We use ode45 to integrate the dynamic system for every $0.01s$ time-interval and quadprog to solve QP. Initial values of states are $x(0)=-2.5m,y(0)=0m,\theta(0)=0,v(0)=0.5m/s^{2},a_{1}(0)=1001.$ The control bound is $[-10,-5]^{T}\le \boldsymbol{u} \le [10,5]^{T}.$
We first assess AVCLBF's adaptivity to finite-time reachability by varying the target circle's radius. For safety, the robot must avoid two overlapping solid red circles and remain within a hollow red circle.
 The parameters related to HOCBFs are $x_{1,o}=x_{2,o}=y_{3,o}=0m, x_{3,o}=1m, y_{1,o}=-y_{2,o}=0.5m, r_{1,o}=r_{2,o}=1m, r_{3,o}=4.5m, k_{1}=k_{2}=1.$ For finite-time reachability requirements, the robot needs to reach hollow green circles within 5 seconds.
The parameters related to AVCLBFs are $x_{1,d}=x_{2,d}=x_{3,d}=3m, y_{1,d}=y_{2,d}=y_{3,d}=0m, r_{1,d}=1m,r_{2,d}=1.5m,r_{3,d}=0.5m.$
Other parameters are $l_{1}=l_{2}=1, W_{1}=1000, a_{1,w}=1000,\epsilon=10^{-10}.$ This is a particularly challenging example because the robot initially heads straight towards obstacles, and the area it needs to reach is directly to the right of the obstacles. In Fig. \ref{fig:diff radius}, the robot can safely enter the target area within 5 seconds, regardless of the radius of the area. Fig. \ref{fig:diff radius2} shows more details about state reachability over time represented by $h(\boldsymbol{x}).$ Note that we set $q=\frac{1}{4}, c(t)=5t-0.2$ for solid lines ($C_{\int}(t)=C_{\int}(0)+\frac{5}{2}t^{2}-0.2t$). To test the adaptivity of AVCLBF to hyperparameters, we first let $q$ take different values. 
  It is shown in Figs. \ref{fig:diff q1}, \ref{fig:diff q12} that the robot can always safely enter the target area within 5 seconds. Secondly, we change the the function of $c(t)$ into $c(t)=2t^{2}-2\ (C_{\int}(t)=C_{\int}(0)+\frac{2}{3}t^{3}-2t$), and it is shown in Figs. \ref{fig:diff radius}, \ref{fig:diff  radius2} that the robot can still satisfy safety and finite-time reachability requirements (depicted by dashed line).
  
We compare our proposed AVCLBFs with the state of the art time-varying CBFs and HOCLBFs. In Fig. \ref{fig:narrow passage}, the robot tries to drive through a narrow passage. The parameters are set as follows: for the map, $x_{1,o}=x_{2,o}=0m,y_{1,o}=-y_{2,o}=1m,r_{1,o}=r_{2,o}=1m;$ for AVCLBF, $k_{1}=k_{2}=1, l_{1}=0.5,l_{2}=1, q=\frac{2}{3}, c(t)=4t-1,$ other parameters are the same as before; for HOCLBF1, $k_{1}=k_{2}=1, c_{1}=c_{2}=5, q_{1}=q_{2}=\frac{1}{3};$ for HOCLBF2, $k_{1}=k_{2}=1, c_{1}=c_{2}=5, q_{1}=q_{2}=\frac{1}{5};$ for time-varying CBF, $k_{1}=k_{2}=1, l_{1}=l_{2}=0.5, T=5s,r_{d_{0}}=6m,r_{d_{T}}=r_{d}.$ $x_{d},y_{d},r_{d}$ for three methods are set as $3m, 0m, 1m$ respectively. Only AVCLBF enables the robot to navigate the narrow passage and meet all requirements, whereas benchmark methods, shown as dashed lines, fail midway, marked by crosses. Fig. \ref{fig:diff theta} examines how initial angles impact benchmark performance. For time-varying CBF, we change the initial angle of robot into $\theta(0)=0,\frac{\pi}{4},\frac{\pi}{2},$ the robot cannot find a feasible path to the target area (depicted by dashed black lines). For HOCLBF, we set $\theta(0)=0,\frac{\pi}{6},\frac{\pi}{4}$ for HOCLBF1, and $\theta(0)=0,-\frac{\pi}{6},-\frac{\pi}{4}$ for HOCLBF2. We see that the robot can only find a feasible path when the absolute values of angles are large (the cases are not very challenging). However for AVCLBF, we set $k_{1}=k_{2}=1, l_{1}=l_{2}=5,\theta(0)=0,q=\frac{1}{3},c(t)=1-3t+2t^2,$ and the robot finally finds the ideal path (solid blue trajectory). Moreover, we make $x_{3,o}=0.5m,y_{3,o}=0m$ and reduce the radius of the outer safety circle to $r_{3,o}=0.31m$ (dashed red circle)  to create a more constricted safe space, AVCLBF still works (dashed blue trajectory). Note that in this case, some parameters of AVCLBF are set the same as those of HOCLBF1. We conclude the proposed AVCLBF outperforms the benchmarks in terms of adaptivity and feasibility.

\begin{figure*}[t]
    \vspace{3mm}
    \centering
    \begin{subfigure}[t]{0.24\linewidth}
        \centering
        \includegraphics[width=1\linewidth]{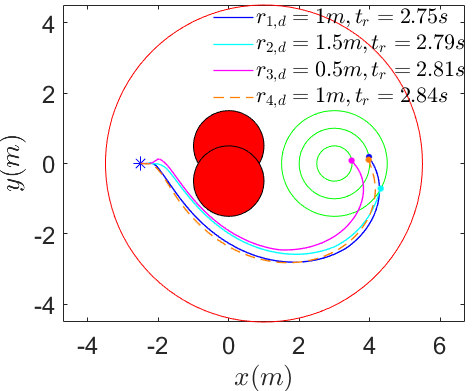}
        \caption{$q=\frac{1}{4},$ solid lines: $c(t)=5t-0.2,$ dashed line: $c(t)=2t^{2}-2.$}
        \label{fig:diff radius}
    \end{subfigure}
    \begin{subfigure}[t]{0.24\linewidth}
        \centering
        \includegraphics[width=1\linewidth]{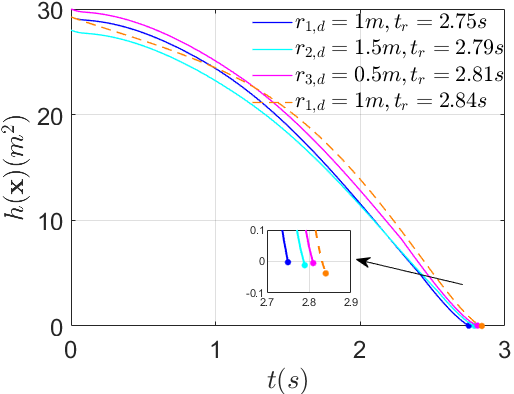}
        \caption{$q=\frac{1}{4},$ solid lines: $c(t)=5t-0.2,$ dashed line: $c(t)=2t^{2}-2.$}
        \label{fig:diff radius2}
    \end{subfigure}  
    \begin{subfigure}[t]{0.24\linewidth}
        \centering
        \includegraphics[width=1\linewidth]{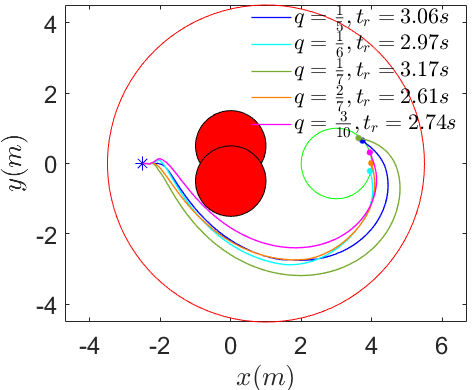}
        \caption{Solid lines: $c(t)=5t-0.2.$}
        \label{fig:diff q1}
    \end{subfigure}
    \begin{subfigure}[t]{0.24\linewidth}
        \centering
        \includegraphics[width=1\linewidth]{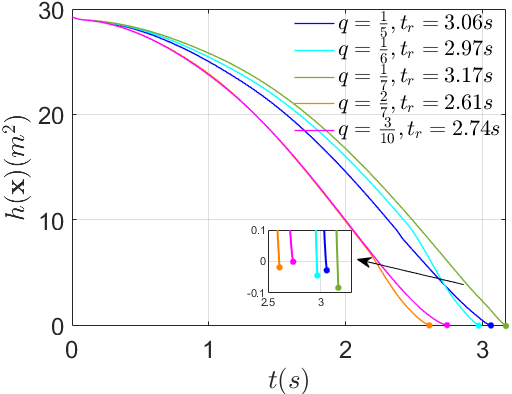}
        \caption{Solid lines: $c(t)=5t-0.2.$}
        \label{fig:diff q12}
    \end{subfigure}
    \caption{Closed-loop trajectories (in Fig. \ref{fig:diff radius}, Fig. \ref{fig:diff q1}) and curves of $h(\boldsymbol{x})$ over time (in Fig. \ref{fig:diff radius2}, Fig. \ref{fig:diff q12}) with controllers AVCLBF-HOCBF-QP: several safe closed-loop trajectories starting at the same location (depicted by blue asterisk) terminates within the target areas (depicted by green circles) before 5 seconds. Their termination time is represented by $t_{r}.$
    } 
    \label{fig:open-closed-loop}
\end{figure*}

\begin{figure}[ht]
    \centering
    \includegraphics[scale=0.38]{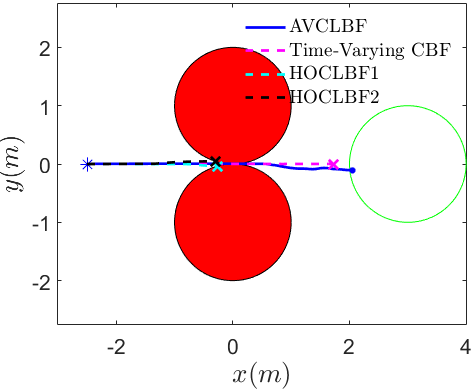}
    \caption{Robot controlled by AVCLBF (solid line) can pass the narrow passage and reach the green circle at $t_{r}=1.99s.$ Benchmark methods (dashed lines) fail at midway.}
    \label{fig:narrow passage}
\end{figure} 

\begin{figure}[ht]
    \centering
    \includegraphics[scale=0.38]{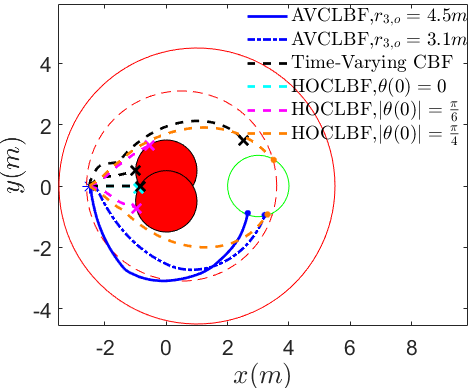}
    \caption{Robot controlled by AVCLBF (blue lines) can avoid obstacles and reach the green circle before 5 seconds ($t_{r}=2.67s$ for solid line and $t_{r}=3.31s$ for dashed line) given $\theta(0)=0.$ Benchmarks (lines except blue) fail at midway if $\left | \theta(0) \right | $ is small. The solid red circle confines all trajectories except the dashed blue one, while the dashed red circle confines only the dashed blue trajectory. }
    \label{fig:diff theta}
\end{figure} 

\section{Conclusion and Future Work}
\label{sec:Conclusion and Future Work}
We proposed Auxiliary-Variable Adaptive Control Lyapunov Barrier Functions (AVCLBFs) for the design of optimal controllers in safety-critical applications with spatio-temporal constraints. We showed that AVCLBFs are superior in terms of adaptivity and feasibility when compared with HOCLBFs and time-varying CBFs. We validated the proposed AVCLBFs approach by applying it to a model of a unicycle robot. Our proposed method generated a safe trajectory for the robot, which terminated at the target area within the specified time under numerous parameter settings and map configurations. One limitation is that the feasibility of the optimization and spatio-temporally constrained safety are not always guaranteed at the same time in the whole state space. We will address this limitation in future work by designing a feasibility-guaranteed AVCLBFs method. 
\bibliographystyle{IEEEtran}
\balance
\bibliography{references.bib}

\begin{thebibliography}{10}
\providecommand{\url}[1]{#1}
\csname url@samestyle\endcsname
\providecommand{\newblock}{\relax}
\providecommand{\bibinfo}[2]{#2}
\providecommand{\BIBentrySTDinterwordspacing}{\spaceskip=0pt\relax}
\providecommand{\BIBentryALTinterwordstretchfactor}{4}
\providecommand{\BIBentryALTinterwordspacing}{\spaceskip=\fontdimen2\font plus
\BIBentryALTinterwordstretchfactor\fontdimen3\font minus \fontdimen4\font\relax}
\providecommand{\BIBforeignlanguage}[2]{{%
\expandafter\ifx\csname l@#1\endcsname\relax
\typeout{** WARNING: IEEEtran.bst: No hyphenation pattern has been}%
\typeout{** loaded for the language `#1'. Using the pattern for}%
\typeout{** the default language instead.}%
\else
\language=\csname l@#1\endcsname
\fi
#2}}
\providecommand{\BIBdecl}{\relax}
\BIBdecl

\bibitem{tee2009barrier}
K.~P. Tee, S.~S. Ge, and E.~H. Tay, ``Barrier lyapunov functions for the control of output-constrained nonlinear systems,'' \emph{Automatica}, vol.~45, no.~4, pp. 918--927, 2009.

\bibitem{boyd2004convex}
S.~Boyd, S.~P. Boyd, and L.~Vandenberghe, \emph{Convex optimization}.\hskip 1em plus 0.5em minus 0.4em\relax Cambridge university press, 2004.

\bibitem{aubin2011viability}
J.-P. Aubin, A.~M. Bayen, and P.~Saint-Pierre, \emph{Viability theory: new directions}.\hskip 1em plus 0.5em minus 0.4em\relax Springer Science \& Business Media, 2011.

\bibitem{prajna2007framework}
S.~Prajna, A.~Jadbabaie, and G.~J. Pappas, ``A framework for worst-case and stochastic safety verification using barrier certificates,'' \emph{IEEE Transactions on Automatic Control}, vol.~52, no.~8, pp. 1415--1428, 2007.

\bibitem{panagou2013multi}
D.~Panagou, D.~M. Stipanovi{\v{c}}, and P.~G. Voulgaris, ``Multi-objective control for multi-agent systems using lyapunov-like barrier functions,'' in \emph{52nd IEEE Conference on Decision and Control}, 2013, pp. 1478--1483.

\bibitem{wang2016multi}
L.~Wang, A.~D. Ames, and M.~Egerstedt, ``Multi-objective compositions for collision-free connectivity maintenance in teams of mobile robots,'' in \emph{2016 IEEE 55th Conference on Decision and Control (CDC)}, 2016, pp. 2659--2664.

\bibitem{glotfelter2017nonsmooth}
P.~Glotfelter, J.~Cort{\'e}s, and M.~Egerstedt, ``Nonsmooth barrier functions with applications to multi-robot systems,'' \emph{IEEE control systems letters}, vol.~1, no.~2, pp. 310--315, 2017.

\bibitem{ames2016control}
A.~D. Ames, X.~Xu, J.~W. Grizzle, and P.~Tabuada, ``Control barrier function based quadratic programs for safety critical systems,'' \emph{IEEE Transactions on Automatic Control}, vol.~62, no.~8, pp. 3861--3876, 2016.

\bibitem{ames2012control}
A.~D. Ames, K.~Galloway, and J.~W. Grizzle, ``Control lyapunov functions and hybrid zero dynamics,'' in \emph{2012 IEEE 51st IEEE Conference on Decision and Control (CDC)}, 2012, pp. 6837--6842.

\bibitem{nguyen2016exponential}
Q.~Nguyen and K.~Sreenath, ``Exponential control barrier functions for enforcing high relative-degree safety-critical constraints,'' in \emph{2016 American Control Conference (ACC)}, 2016, pp. 322--328.

\bibitem{xiao2021high}
W.~Xiao and C.~Belta, ``High-order control barrier functions,'' \emph{IEEE Transactions on Automatic Control}, vol.~67, no.~7, pp. 3655--3662, 2021.

\bibitem{isaly2020zeroing}
A.~Isaly, B.~C. Allen, R.~G. Sanfelice, and W.~E. Dixon, ``Zeroing control barrier functions for safe volitional pedaling in a motorized cycle,'' \emph{IFAC-PapersOnLine}, vol.~53, no.~5, pp. 218--223, 2020.

\bibitem{khazoom2022humanoid}
C.~Khazoom, D.~Gonzalez-Diaz, Y.~Ding, and S.~Kim, ``Humanoid self-collision avoidance using whole-body control with control barrier functions,'' in \emph{2022 IEEE-RAS 21st International Conference on Humanoid Robots (Humanoids)}, 2022, pp. 558--565.

\bibitem{cavorsi2022multi}
M.~Cavorsi, B.~Capelli, L.~Sabattini, and S.~Gil, ``Multi-robot adversarial resilience using control barrier functions,'' in \emph{Robotics: Science and Systems}, 2022.

\bibitem{lindemann2018control}
L.~Lindemann and D.~V. Dimarogonas, ``Control barrier functions for signal temporal logic tasks,'' \emph{IEEE control systems letters}, vol.~3, no.~1, pp. 96--101, 2018.

\bibitem{liu2021recurrent}
W.~Liu, N.~Mehdipour, and C.~Belta, ``Recurrent neural network controllers for signal temporal logic specifications subject to safety constraints,'' \emph{IEEE Control Systems Letters}, vol.~6, pp. 91--96, 2021.

\bibitem{lindemann2019control}
L.~Lindemann and D.~V. Dimarogonas, ``Control barrier functions for multi-agent systems under conflicting local signal temporal logic tasks,'' \emph{IEEE control systems letters}, vol.~3, no.~3, pp. 757--762, 2019.

\bibitem{garg2019control}
K.~Garg and D.~Panagou, ``Control-lyapunov and control-barrier functions based quadratic program for spatio-temporal specifications,'' in \emph{2019 IEEE 58th Conference on Decision and Control (CDC)}.\hskip 1em plus 0.5em minus 0.4em\relax IEEE, 2019, pp. 1422--1429.

\bibitem{garg2022fixed}
K.~Garg, E.~Arabi, and D.~Panagou, ``Fixed-time control under spatiotemporal and input constraints: A quadratic programming based approach,'' \emph{Automatica}, vol. 141, p. 110314, 2022.

\bibitem{xiao2021high2}
W.~Xiao, C.~A. Belta, and C.~G. Cassandras, ``High order control lyapunov-barrier functions for temporal logic specifications,'' in \emph{2021 American Control Conference (ACC)}.\hskip 1em plus 0.5em minus 0.4em\relax IEEE, 2021, pp. 4886--4891.

\bibitem{liu2023iterative}
S.~Liu, J.~Zeng, K.~Sreenath, and C.~A. Belta, ``Iterative convex optimization for model predictive control with discrete-time high-order control barrier functions,'' in \emph{2023 American Control Conference (ACC)}, 2023, pp. 3368--3375.

\bibitem{gurriet2018online}
T.~Gurriet, M.~Mote, A.~D. Ames, and E.~Feron, ``An online approach to active set invariance,'' in \emph{2018 IEEE Conference on Decision and Control (CDC)}, 2018, pp. 3592--3599.

\bibitem{singletary2019online}
A.~Singletary, P.~Nilsson, T.~Gurriet, and A.~D. Ames, ``Online active safety for robotic manipulators,'' in \emph{2019 IEEE/RSJ International Conference on Intelligent Robots and Systems (IROS)}, 2019, pp. 173--178.

\bibitem{chen2021backup}
Y.~Chen, M.~Jankovic, M.~Santillo, and A.~D. Ames, ``Backup control barrier functions: Formulation and comparative study,'' in \emph{2021 60th IEEE Conference on Decision and Control (CDC)}, 2021, pp. 6835--6841.

\bibitem{squires2018constructive}
E.~Squires, P.~Pierpaoli, and M.~Egerstedt, ``Constructive barrier certificates with applications to fixed-wing aircraft collision avoidance,'' in \emph{2018 IEEE Conference on Control Technology and Applications (CCTA)}, 2018, pp. 1656--1661.

\bibitem{breeden2021high}
J.~Breeden and D.~Panagou, ``High relative degree control barrier functions under input constraints,'' in \emph{2021 60th IEEE Conference on Decision and Control (CDC)}, 2021, pp. 6119--6124.

\bibitem{xiao2021adaptive}
W.~Xiao, C.~Belta, and C.~G. Cassandras, ``Adaptive control barrier functions,'' \emph{IEEE Transactions on Automatic Control}, vol.~67, no.~5, pp. 2267--2281, 2021.

\bibitem{liu2023auxiliary}
S.~Liu, W.~Xiao, and C.~A. Belta, ``Auxiliary-variable adaptive control barrier functions for safety critical systems,'' in \emph{2023 62th IEEE Conference on Decision and Control (CDC)}, 2023.

\bibitem{xiao2022sufficient}
W.~Xiao, C.~A. Belta, and C.~G. Cassandras, ``Sufficient conditions for feasibility of optimal control problems using control barrier functions,'' \emph{Automatica}, vol. 135, p. 109960, 2022.

\bibitem{liu2023feasibility}
S.~Liu, W.~Xiao, and C.~A. Belta, ``Feasibility-guaranteed safety-critical control with applications to heterogeneous platoons,'' \emph{arXiv preprint arXiv:2310.00238}, 2023.

\bibitem{garg2020prescribed}
K.~Garg, E.~Arabi, and D.~Panagou, ``Prescribed-time convergence with input constraints: A control lyapunov function based approach,'' in \emph{2020 American Control Conference (ACC)}, 2020, pp. 962--967.

\bibitem{Khalil:1173048}
\BIBentryALTinterwordspacing
H.~K. Khalil, \emph{Nonlinear systems; 3rd ed.}\hskip 1em plus 0.5em minus 0.4em\relax Upper Saddle River, NJ: Prentice-Hall, 2002, the book can be consulted by contacting: PH-AID: Wallet, Lionel. [Online]. Available: \url{https://cds.cern.ch/record/1173048}
\BIBentrySTDinterwordspacing

\end{thebibliography}
\end{document}